\documentclass[11pt]{amsart}
\allowdisplaybreaks

\usepackage{amsfonts}
\usepackage{amsmath}
\usepackage{amssymb}
\usepackage{amsthm}
\usepackage{graphicx} \usepackage{enumerate} \usepackage{multicol}
\usepackage{mathrsfs} \usepackage[all,cmtip]{xy}
\usepackage{enumerate}
\usepackage{xcolor}

\usepackage{mathtools}
\mathtoolsset{showonlyrefs}

\usepackage{color}
\newcommand\blue[1]{\textcolor{blue}{#1}}
\newcommand\red[1]{\textcolor{red}{#1}}
\newcommand\green[1]{\textcolor{green}{#1}}

\usepackage[colorlinks,linkcolor=blue,citecolor=blue,pagebackref,hypertexnames=false, breaklinks]{hyperref}
\newtheorem{theorem}{Theorem}[section]

\newtheorem{proposition}[theorem]{Proposition}

\newtheorem{corollary}[theorem]{Corollary}
\newtheorem{lemma}[theorem]{Lemma}
\newtheorem{remark}[theorem]{Remark}
\newtheorem{example}[theorem]{Example}
\newtheorem{examples}[theorem]{Examples}
\newtheorem{foo}[theorem]{Remarks}

\newcommand\transpose{%
  {\mathchoice
    {\raisebox{.45ex}{$\displaystyle{\intercal}$}}
    {\raisebox{.45ex}{$\textstyle{\intercal}$}}
    {\raisebox{.30ex}{$\scriptstyle{\intercal}$}}  
    {\raisebox{.23ex}{$\scriptscriptstyle{\intercal}$}}}
  }

\newcommand{\calE}{\mathcal{E}}
\newcommand{\calF}{\mathcal{F}}
\newcommand{\calH}{\mathcal{H}}
\newcommand{\calL}{\mathcal{L}}
\newcommand{\Ent}{\textrm{Ent}}
\newcommand{\bH}{\mathbb{H}}
\newcommand{\M}{\mathbb{M}}
\newcommand{\bM}{\mathbb{M}}
\newcommand{\bR}{\mathbb{R}}
\newcommand{\calO}{\mathcal{O}}
\newcommand{\calV}{\mathcal{V}}
\newcommand{\Dh}{\Delta_\mathcal{H}}
\newcommand{\Dhe}{\Delta_{\mathcal{H},\varepsilon}}
\newcommand{\calR}{\mathcal{R}}
\newcommand{\dR}{\mathbb{R}}
\newcommand{\cP}{\mathcal{P}}
\newcommand{\dN}{\mathbb{N}}
\newcommand{\R}{\mathbb{R}}
\newcommand{\V}{\mathcal{V}}
\newcommand{\ch}{\mathcal{H}}
\newcommand{\ee}{\ell}
\newcommand{\J}{\mathfrak J}

\newcommand{\G}{G}
\newcommand{\Ga}{\Gamma}
\newcommand{\He}{\mathbb H}

\newcommand{\DV}{D_{Vill}}
\newcommand{\al}{\alpha}
\newcommand{\be}{\beta}
\newcommand{\ld}{\lambda}
\newcommand{\ve}{\varepsilon}
\newcommand{\ga}{\gamma}
\newcommand{\f}{e^{\lambda f}}
\newcommand{\pl}{e^{\lambda \psi}}

\newcommand{\Var}{\textrm{Var}}

\newcommand{\rmB}{\mathrm{B}}

\newcommand{\scrA}{\mathscr{A}}

\DeclareMathOperator{\Ann}{Ann}
\DeclareMathOperator{\Sym}{Sym}
\DeclareMathOperator{\End}{End}
\DeclareMathOperator{\id}{id}
\DeclareMathOperator{\inc}{inc}
\DeclareMathOperator{\pr}{pr}
\DeclareMathOperator{\rank}{rank}
\DeclareMathOperator{\spn}{span}
\DeclareMathOperator{\image}{image}
\DeclareMathOperator{\tr}{tr}

\DeclareMathOperator{\Ad}{Ad}
\DeclareMathOperator{\ad}{ad}
\DeclareMathOperator{\GL}{GL}
\DeclareMathOperator{\gl}{\mathfrak{gl}}
\DeclareMathOperator{\SO}{SO}
\DeclareMathOperator{\SU}{SU}
\DeclareMathOperator{\so}{\mathfrak{so}}
\DeclareMathOperator{\Ort}{O}
\DeclareMathOperator{\ort}{\mathfrak{o}}
\DeclareMathOperator{\su}{\mathfrak{su}}
\DeclareMathOperator{\Spin}{Spin}
\DeclareMathOperator{\Aut}{Aut}
\DeclareMathOperator{\Der}{Der}

\newcommand{\eps}{\varepsilon}
\DeclareMathOperator{\Lie}{Lie}
\DeclareMathOperator{\dil}{\delta}

\DeclareMathOperator{\Ric}{Ric}
\DeclareMathOperator{\Hol}{Hol}
\DeclareMathOperator{\Hess}{Hess}
\DeclareMathOperator{\ptr}{/\! \! /}

\newcommand\blank{{\kern.8pt\cdot\kern.8pt}}

\renewcommand{\theequation}{\arabic{equation}}

\title[Ornstein-Uhlenbeck]{Convergence to equilibrium for hypoelliptic non-symmetric Ornstein-Uhlenbeck type operators}
\author{Fabrice Baudoin, Michel Bonnefont, Li Chen}

\thanks{F.B. is supported in part by NSF Grant  DMS-1901315 and the Simons Foundation. M.B. is partly supported by the ANR project RAGE ANR-18-CE40-0012.}

\date{\today}

\begin{document}

\begin{abstract}
We study a generalized curvature dimension inequality which is suitable for subelliptic  Ornstein-Uhlenbeck type operators and deduce convergence to equilibrium in the $L^2$ and  entropic sense. The main difficulty is that the operators we consider may not be symmetric. Our results apply in particular to Ornstein-Uhlenbeck operators on two-step Carnot groups. 
 \end{abstract}

\maketitle

\tableofcontents

\section{Introduction}
In order to motivate the present  work  and before turning to the sub-elliptic situation,  we recall   well known facts for the corresponding  elliptic operators. Given a Riemannian manifold $ \M$ with metric $g$  and equipped with a smooth measure $d\mu= e^{-V} d\, \mathrm{vol}$, one can naturally consider the following diffusion operator:
\[
\Delta_g f - \nabla_g V \cdot \nabla_g f,
\]
for $f\in  C_0^\infty (\M)$,  where $\Delta_g$ denotes the Laplace-Betrami operator. 
This operator  is symmetric with respect to the measure $\mu$ and one can construct the associated semi-group $P_t$.  We refer to e.g. \cite{BGL} for more details.

Among them,  the most typical example may be  the Ornstein-Uhlenbeck semigroup on $ \R^n$  which   admits  the standard Gaussian measure as invariant and reversible measure  semigroup. Its generator thus  reads as  $\Delta f -x \cdot \nabla f$.

In this setting, an important question is to study the long time behaviour of  the semigroup, and if it is the case to study its convergence to equilibrium.

In the Riemannian case, an important answer was given by Bakry and Emery \cite{BE} with the introduction of  their famous curvature-dimension criterion $CD(\rho,\infty)$. This criterion  reads shortly $''\Gamma_2\geq \rho \Gamma''$;    where $\Gamma$,  the carr\'e du champ operator,  and $\Gamma_2$ , its  iteration, can be canonically defined from the diffusion opertor $L$  (see Section \ref{sec:framework} for precise definitions).  

Now if the criterion $CD(\rho,\infty)$ is satisfied with $\rho >0$, both the usual Poincar\'e inequality and the Log-Sobolev inequalities are satisfied which respectively gives an exponential convergence to equilibrium in $L^2$ and in entropy for the semi-group. 

In the context of weighted Riemannian manifold, the $CD(\rho,\infty)$  criterion is satisfied if and only if $Ric + \nabla \nabla V \geq \rho$.

In the case of the classical Ornstein-Uhlenbeck, the Bakry-\'Emery calculus actually provides the optimal Poincar\'e and Log-Sobolev inequality, and thus the optimal rates of convergence.

The goal of the present work is to investigate similar results in the subelliptic situation. 
In this situation, for example the seminal  example of the sub-Laplacian on the  Heisenberg group, the Bakry-Emery criterion fails. However, recently the first author with Nicola Garofalo introduced and studied systematically  a generalized curvature criterion adapted to the sub-Riemannian setting \cite{BG}. They work  mainly in  the unweighted case and consider mainly sub-Laplacians  $\Delta_\mathcal{H}$

The present work focusses on the study of operators of the form: $\Delta_\mathcal{H} -X$ where  $\Delta_\mathcal{H}$ is a sub-Laplacian and $X$  a smooth vector field.
One difficulty that arises in this context is that one may have naturally to consider non-symmetric operators.

Indeed, in the case of the Heisenberg group, it is natural to consider the heat kernel  $p_t$ as the generalization of the Gaussian measure. 
As in the case of the Riemannian weighted manifold, and as it was done in \cite{BHT}, one can consider the symmetric operator:
\[
\Delta_\mathcal{H}  - \nabla_\mathcal{H}  \ln p_1 \cdot \nabla_\mathcal{H}.
\]

But since, the heat kernel is only known from an oscillatory integral, the Gaveau formula,   it does not behave well with respect of curvature-dimension  criterion.    In some sense, it is more natural to consider  another operator  $\Delta_\mathcal{H}  - 2D$ where $D$ is the dilation vector field, which still admits the heat kernel $p_1$ for invariant measure but which is not symmetric. This operator is called  the Ornstein-Uhlenbeck operator on the Heisenberg group.  It has already been studied in \cite{BHT,LP}.  Actually, such an operator is well defined on all Carnot groups.

We introduce a class of sub-elliptic diffusions  which encompasses the case  of these  Ornstein-Uhlenbeck semigroups on Carnot groups of rank 2,  and which satisfies  a  generalized curvature criterion holding under natural geometric assumptions.  We then  investigate the long time behaviour of their associated semi-group.

For this, we consider the setting of sub-Laplacians on  totally geodesic foliated Riemmanian manifolds with a drift whose horizontal part is basic (see Section \ref{sec:foliation} for more details).  With natural assumption,  the generalized curvature criterion holds and we establish  exponential  $L^2$   and entropic convergences.  As said before, some  difficulties arise since the operators we consider are not symmetric. We only obtain modified Poincar\'e and Log-Sobolev inequalities with an elliptic gradient in the energy.

Some comments must be done. First, in the   specific case  of   Ornstein-Uhlenbeck semi-groups on Carnot groups, the semi-group   $Q_t$ can be  expressed by the Mehler formula   as a time change of the heat semigroup $P_t$ associated to the canonical sub-Laplacian by the formula
\[
Q_t(f) (x)= P_{1-e^{-2t}}(f)\left(\dil_{e^{-t} } x \right)
\] 
 where $\dil$ is the dilation on the Carnot group. Another way to proceed is thus to study the symmetric heat semi-group  $P_t$ through the generalized curvature-dimension criterion  for the sub-Laplacian of the Carnot group as done in \cite{BG} and to transfer the results with the Melher's formula.  This  would lead to the same results as the ones  obtained here in this specific situation.

Secondly,  in the case of the Heisenberg group or  more generally in the particular   case of  Carnot groups of type $H$, the usual log-Sobolev  for the heat kernel is known \cite{HQLi,BBBC,HZ}.   For the moment, with the generalized curvature-dimension criterion, we are unable
to reach such a result.

Finally, we mention the work \cite{BJFA17} where some different hypoelliptic operators arising from the kinetic Fokker-Planck equations are treated.

The paper is organized as follows. First in Section  \ref{sec:framework}, we present the framework and the hypothesis under which we work.  We introduce the geometric setting  and  present examples  of diffusion operators satisfying the desired curvature  condition with a focus on the examples of Ornstein-Uhlenbeck semigroups on Carnot groups.  Section \ref{sec:L2} is devoted to the convergence in $L^2$ of the semi-groups whereas the final section treats the more difficult case of the entropic convergence.

%
%
%
%
%
%

\section{Framework and Examples}\label{sec:framework}

\subsection{Framework}

Let $\M$ be a smooth connected finite dimensional manifold endowed with a smooth probability measure $\mu$. 
Let $L$ be a locally subelliptic second-order diffusion operator on $\M$ and assume that $\mu$ is an invariant measure of $L$, i.e. for every $f \in C_0^\infty(\M)$
\[
\int_\M Lf d\mu=0.
\]

 We indicate by $\Gamma (f):=\Gamma(f,f)$  the associated carr\'e du champ defined by
\[
\Gamma(f,g)=\frac{1}{2} ( L (fg) -gLf -fL g),\quad \forall f,g\in C^{\infty}(\M)
\]
In addition, we assume that $\M$ is endowed with another smooth symmetric bilinear differential form, denoted by $\Gamma^Z$, satisfying for $f,g, h \in C^\infty(\M)$,
\[
\Gamma^Z (fg,h)=g\Gamma^Z (f,h)+f\Gamma^Z (g,h)
\]
and
\[\Gamma^Z(f,f)\geq 0.\]
We assume that  
\[
d(x,y)=\sup\{|f(x)-f(y)|: f\in C^{\infty}(\M), \, \|\Gamma(f)\|_{\infty}\le 1\}, \quad \forall x,y\in \M.
\]
is a genuine distance.
In a similar manner, for any $\ve >0$, we define $d_{\ve}$ to be the distance associated to the operator $\Gamma_{\ve}:=\Gamma+\ve \Gamma^Z$.

Now consider the iterations of $\Gamma$ and $\Gamma^Z$ which are defined by
\[
\Gamma_2(f,g)=\frac{1}{2} ( L(\Gamma(f,g)) -\Gamma(g,L f)-\Gamma(f,L g)),
\]
\[
\Gamma^Z_2(f,g)=\frac{1}{2} ( L (\Gamma^Z(f,g)) -\Gamma^Z(g,L f)-\Gamma^Z(f,L g)).
\]

Throughout the paper, we make the following assumptions.
\begin{itemize}
\item[({\bf A1})] There exists a nice Lyapounov function $W\ge 1$ such that $\Gamma(W)+\Gamma^Z(W)\le C W^2$, $LW\le CW$,
and  $\{W\le m\}$ is compact for every $m$.

\item[({\bf A2})]  
For any $f \in C^\infty(\bM)$ one has
\[
\Gamma(f, \Gamma^Z(f))=\Gamma^Z( f, \Gamma(f)).
\] 

\item[({\bf A3})] The following generalized curvature dimension condition holds: there exist $\rho_1,\rho_2,\rho_3,\kappa >0$ such that
\[
\Gamma_2(f) +\ve \Gamma_2^Z (f) \ge \left( \rho_1-\frac{\kappa}{\ve}\right) \Gamma(f) + \left(\rho_2 +\rho_3 \ve \right) \Gamma^Z(f)
\]
for every $f\in C^{\infty}(M)$ and $\ve>0$.
 \end{itemize}

Under this assumption ({\bf A1}), the Markov semigroup $(Q_t)_{t\ge 0}$ with generator $L$ uniquely solves the heat equation in $L^{\infty}$. Moreover, consider a smooth function $h: \mathbb R_{\ge 0} \to \mathbb R$ such that $h=1$ on $[0,1]$ and $h=0$  on $[2,\infty)$. Denote $h_n=h\left(\frac{W}{n}\right)$ and consider the compactly supported diffusion operator $L_n=h_n^2L$. Since $L_n$ is compactly supported, a Markov semigroup $Q_t^n$ can be constructed as the unique bounded solution of  $\frac{\partial Q_t^n f}{\partial t}=L_n Q_t^n f$, with  $Q^n_0f=f\in L^{\infty}$. Then for every bounded function $f$, we have  $Q_t^n f \to Q_t f$, as $n\to \infty$. We refer to \cite{W12} for more details. Under the assumptions  ({\bf A1}) and  ({\bf A3}) we have the following basic result:

\begin{lemma}\label{invariance semigroup}
For every $f \in C^\infty_0(\M)$ and $t \ge 0$,
\[
\int_\M Q_t f d\mu =\int_\M f d\mu.
\]
\end{lemma}

\begin{proof}
As above, consider a smooth function $h: \mathbb R_{\ge 0} \to \mathbb R$ such that $h=1$ on $[0,1]$ and $h=0$  on $[2,\infty)$ and denote $h_n=h\left(\frac{W}{n}\right)$. One has
\begin{align*}
\int_\M (Q_t f -f )h_n d\mu&=\int_0^t \int_\M (LQ_s f) h_n d\mu ds \\
 &=\int_0^t \int_\M \left(L(Q_s f h_n)-Q_sf Lh_n-2\Gamma(Q_sf,h_n) \right)d\mu ds \\
 &=-\int_0^t \int_\M Q_sf Lh_nd\mu ds -2\int_0^t \int_\M\Gamma(Q_sf,h_n) d\mu ds.
\end{align*}
The term $\int_0^t \int_\M Q_sf Lh_nd\mu ds$ goes to 0 when $n \to +\infty$ from the definition of $h_n$ and the assumption ({\bf A1}) on $W$. For the second term, one observes that $| \Gamma(Q_sf,h_n)| \le \sqrt{\Gamma(Q_sf )} \sqrt{\Gamma(h_n )}$. The quantity $\sqrt{\Gamma(Q_sf )}$ can be controlled using Proposition \ref{prop:dec-gammaZ-exp} below (this relies on  ({\bf A3})) and $ \sqrt{\Gamma(h_n )}$ goes to 0 from the definition of $h_n$ and the assumption ({\bf A1}) on $W$. One concludes
\[
\lim_{n \to +\infty} \int_\M (Q_t f -f )h_n d\mu=0.
\]
Using dominated convergence, this implies
\[
\int_\M Q_t f d\mu =\int_\M f d\mu.
\]
\end{proof}

A corollary which is useful  is the following.

\begin{lemma}\label{lem:phi}
Let $\varphi:\R \to \R$ (or  $\R_+ \to \R$) be convex  such that $\varphi(0)=0$.  Then for any $f\in C_0^{\infty}(\M)$  \begin{equation}\label{e:phi}
 \int_\M   \varphi(Q_t f) d\mu \leq \int_\M   \varphi( f) d\mu.
 \end{equation}
 In particular,  the energy $\int_\M (Q_t f)^2 d\mu$ and the entropy $\int_\M Q_t f \ln (Q_t f)d\mu$ are non increasing in $t\geq 0$.
\end{lemma}

\begin{proof}
Since $\varphi$ is convex, from Jensen's inequality one has
\[
\varphi(Q_t f) \le Q_t\varphi( f).
\]
The result follows then from Lemma \ref{invariance semigroup}.
\end{proof}


\subsection{Examples of Ornstein-Uhlenbeck type operators}
In this section, we  study examples of Ornstein-Uhlenbeck operators which satisfy the assumptions of our general framework.

\subsubsection{Heisenberg group}

The motivating and basic example to which our results apply  is the Ornstein-Uhlenbeck operator on the Heisenberg group, as defined in \cite{LP}.

The Heisenberg group is the set 
\[
\bH=\{(x,y,z): x\in \bR, y\in \bR, z\in \bR\}
\]
 endowed with the group law
\[
(x,y,z) \cdot (x',y',z')=\left(x+x',y+y',z+z'+\frac12(xy'-yx')\right).
\]
Consider the left-invariant vector fields: 
\[
X=\partial_{x}-\frac{y}{2}\partial_z,\quad Y=\partial_{y}+\frac{x}{2}\partial_z, \quad Z=\partial_z,
\]
and the Ornstein-Uhlenbeck type  operator: 
$$
L= X^2+Y^2 -2 D =\Dh-2 D,
$$
where $D$ is  the dilation vector field
$$
 D=\frac{1}{2}x \partial_x +\frac{1}{2} y \partial_y +z \partial_z.
$$
Notice that $L$ is not symmetric, but it admits an invariant measure $p_\frac{1}{2}$ where $p_t$ denotes the heat kernel associated to the sub-Laplacian $\Dh=X^2+Y^2$.
More generally, the operator $L=\Dh - \alpha D$ admits $p_\frac{1}{\alpha}$ as invariant measure.

Actually, $D$ can also be written as
$$
D=\frac{1}{2}xX+\frac{1}{2}yY +zZ
$$
and we have the following commutation relations, which are easy to check:
$$
[X,D]=\frac{1}{2}X ,\, [Y,D]=\frac{1}{2}Y, [Z,D]=Z.
$$
Assumption ({\bf A1}) is satisfied with $W=1+(x^2+y^2)^2 +z^2$. Assumption ({\bf A2}) is easily seen to be satisfied and the following proposition shows that  assumption ({\bf A3})  is satisfied.

\begin{proposition}\label{CD-Heisenberg}
Let $f\in C^\infty(\mathbb H)$ and $\varepsilon >0$, then
$$
\Ga_2 (f,f) + \ve \Ga_2^Z (f,f) \geq \left(1-\frac{1}{\ve}\right) \Ga(f,f) + 
\left(2\ve+ \frac{1}{2}\right)(Zf)^2.
$$
\end{proposition}
\begin{proof}
Denote by $\Ga_2^\He$ the iteration operator associated to the sub-Laplacian $\Dh=X^2+Y^2$, that is, 
\[
\Ga_2^\He(f,f)=\frac12(\Dh\Gamma(f,f)-2\Gamma(f, \Dh f)), \quad \forall f\in C^{\infty}(M).
\]
 Then direct computation yields
\begin{align*}
\Ga_2 (f,f) &= \Ga_2^\He(f,f) + X(f)[X,2D](f) + Y(f)[Y,2D](f)\\ 
      &= \Ga_2^\He(f,f) + \Ga(f,f)\\
      &= (X^2f)^2 +(Y^2f)^2 +(XYf)^2+(YXf)^2  -2 (Xf) (YZf) + 2(Yf) (XZf)   + \Ga(f,f) .
\end{align*}
From \cite[Section 2.2]{BG},  one knows that for any $\ve>0$
\[
\Ga_2^\He(f,f) \ge \frac12(Lf)^2+\frac12(Zf)^2 -\ve\Ga(Zf,Zf)-\frac1{\ve}\Gamma(f,f).
\]
Furthermore, 
\begin{align*}
\Ga_2^Z (f,f) &= \Ga(Zf,Zf) +  Zf [L,Z] (f)\\
              &=\Ga(Zf,Zf) + 2 (Zf)^2,
\end{align*}
since $[\Dh,Z]=0$ and $[Z,2D]=2Z$. Finally, we conclude the desired result by collecting the above computations.
\end{proof}

%
%
%




\subsubsection{Ornstein-Uhlenbeck operators on foliated spaces} \label{sec:foliation}

%
%

Let $\M$ be a smooth, connected, complete Riemannian manifold of dimension $n+m$ endowed with a smooth measure $\mu$. We assume that $\bM$ is equipped with a Riemannian foliation $\mathcal{F}$ with bundle like metric $g$ and totally geodesic  $m$-dimensional leaves for which the horizontal distribution is bracket generating and Yang-Mills. 

We define the horizontal gradient $\nabla_\mathcal{H} f$ of a smooth function $f$ as the projection of the Riemannian gradient of $f$ on the horizontal bundle. Similarly, we define the vertical gradient $\nabla_\mathcal{V} f$ of a function $f$ as the projection of the Riemannian gradient of $f$ on the vertical bundle.
The horizontal Laplacian $\Dh$ is the generator of the symmetric pre-Dirichlet form
\[
\mathcal{E}_{\mathcal{H}} (f,g) =\int_\bM \langle \nabla_\mathcal{H} f , \nabla_\mathcal{H} g \rangle_{\mathcal{H}} d\mu_g, \quad f,g \in C_0^\infty(\M)
\]
where $\mu_g$ is the Riemannian volume measure.   We have therefore the following integration by parts formula
\[
\int_\bM \langle \nabla_\mathcal{H} f , \nabla_\mathcal{H} g \rangle_{\mathcal{H}} d\mu_g=-\int_\M f \Dh g d\mu_g=-\int_\M g \Dh f d\mu_g, \quad f,g \in C_0^\infty(\M).
\]
From this convention $\Dh$ is therefore non positive. 

Consider now on $\M$ the following operator:
\[
L=\Dh-X,
\]
where $X$ is a smooth vector field on $\M$. We do not assume that $X$ is a horizontal vector field. However, we will assume that the horizontal part $X_\mathcal{H}$ of $X$ (i.e.,  its projection onto $\mathcal{H}$) is basic. In other words, $X_\mathcal{H}$ satisfies
\[
\nabla_v X_\mathcal{H}=0,
\]
whenever $v$ is a vertical vector. We denote here by $\nabla$ the Bott connection on $\M$ (see \cite{BKW} for further details).
In addition, we assume that $\mu$ is an invariant measure of $L$.
\

We now introduce the following operators defined for $f,g \in C^\infty(\M)$,
\[
\Gamma(f,g)=\frac{1}{2} ( L (fg) -gLf -fL g)=\langle \nabla_\mathcal{H} f , \nabla_\mathcal{H} g\rangle_\mathcal{H}, 
\]
\[
\Gamma^Z (f,g)=\langle \nabla_\mathcal{V} f , \nabla_\mathcal{V} g\rangle_\mathcal{V}.
\]
Their iterations $\Gamma_2$ and $\Gamma^Z_2$ are defined accordingly.
Observe that from \cite{B16ems, BG} one has
\[
\Gamma (f, \Gamma^Z (f))=\Gamma^Z  (f, \Gamma (f)).
\]

We then obtain the following Bochner's type inequality. The geometric  tensors  $J$,  $\mathbf{J}$ and $T$ used below are defined in Section 4.1 in \cite{B16ems}. For conciseness, we refer the reader to this reference for their definitions and basic properties.

\begin{theorem}\label{CD}
Let $f \in C^\infty(\M)$ and $\varepsilon >0$, we have
\begin{align*}
 & \Gamma_2(f,f)+\varepsilon \Gamma^Z_2(f,f) \\
&  \qquad \ge-\frac{1}{4} \mathbf{Tr}_\mathcal{H} (J^2_{\nabla_\mathcal{V} f})+ \ \mathbf{Ric}_{\mathcal{H}} (\nabla_\mathcal{H} f, \nabla_\mathcal{H} f)  +\frac{1}{\varepsilon} \langle \mathbf{J}^2 (\nabla_\mathcal{H} f) , \nabla_\mathcal{H} f \rangle_\mathcal{H} \\
  & \qquad \quad +\langle T(X_\mathcal{H} ,\nabla_\mathcal{H} f), \nabla_\mathcal{V} f \rangle +\ve \langle \nabla_{\nabla_\mathcal{V} f} X_\mathcal{V}, \nabla_\mathcal{V} f \rangle + \langle \nabla_{\nabla_\mathcal{H} f} X_\mathcal{H}, \nabla_\mathcal{H} f \rangle \\
  & \qquad\quad + \langle \nabla_{\nabla_\mathcal{H} f} X_\mathcal{V}, \nabla_\mathcal{V} f \rangle
\end{align*}
\end{theorem}

\begin{proof}
We split $\Gamma_2(f,f)+\varepsilon \Gamma^Z_2(f,f)$ into four parts as follows
\begin{align*}
 \Gamma_2(f,f)+\varepsilon \Gamma^Z_2(f,f) =& \frac12\left( \Delta_\mathcal{H} \Ga(f,f)-2\Ga(f,\Delta_\mathcal{H}f)\right)+ \frac{\varepsilon}2\left( \Delta_\mathcal{H} \Ga^Z(f,f)-2\Ga^Z(f,\Delta_\mathcal{H}f)\right)
 \\ &-\frac12\left(X \Ga(f,f)-2\Ga(f,Xf)\right)-\frac{\varepsilon}2\left(X \Ga^Z(f,f)-2\Ga^Z(f,Xf)\right)
\\ =:&I+ \varepsilon II- III- \varepsilon IV.
\end{align*}

Since the horizontal distribution is Yang-Mills,  it follows from \cite[Theorem 3.1]{BKW} that
\[
I+  \varepsilon II\ge -\frac{1}{4} \mathbf{Tr}_\mathcal{H} (J^2_{\nabla_\mathcal{V} f})+ \ \mathbf{Ric}_{\mathcal{H}} (\nabla_\mathcal{H} f, \nabla_\mathcal{H} f)  +\frac{1}{\varepsilon} \langle \mathbf{J}^2 (\nabla_\mathcal{H} f) , \nabla_\mathcal{H} f \rangle_\mathcal{H}.
\]

We compute $III$ and $IV$ by introducing a local horizontal  and vertical orthonormal frame $\{X_1, \cdots, X_n, Z_1,\cdots, Z_m\}$. In this case we have 
\begin{align*}
III&=\sum_{i=1}^n (X X_i f) X_i f- \sum_{i=1}^n (X_i X f) X_i f=\sum_{i=1}^n [X,X_i]f X_i f
\\ &=
\sum_{i=1}^n (\nabla_X X_i-\nabla_{X_i}X-T(X,X_i))f X_i f.
\end{align*}
Notice that the Bott connection is metric and the covariant derivative of horizontal vector fields is horizontal, then  
\[
\sum_{i=1}^n (\nabla_X X_i)f X_i f=\sum_{i,j=1}^n \langle \nabla_X X_i, X_j\rangle X_j f X_i f=-\sum_{i,j=1}^n \langle  X_i, \nabla_X X_j\rangle X_j f X_i f=0.
\]
Next we compute
\begin{align*}
\sum_{i=1}^n (\nabla_{X_i}X)f X_i f &= \sum_{i,j=1}^n \langle \nabla_{X_i} X_\mathcal{H}, X_j\rangle X_j f X_i f + 
\sum_{i=1}^n \sum_{k=1}^m \langle \nabla_{X_i} X_\mathcal{V}, Z_k\rangle Z_k f X_i f 
\\ &=
\langle \nabla_{\nabla_\mathcal{H} f} X_\mathcal{H}, \nabla_\mathcal{H} f\rangle+
\langle \nabla_{\nabla_\mathcal{H} f} X_\mathcal{V}, \nabla_\mathcal{V} f\rangle.
\end{align*}
Finally observe that $T(X, X_i)=T(X_\mathcal{H},X_i)$ is vertical, then 
\[
\sum_{i=1}^n T(X,X_i)f X_i f = \sum_{i=1}^n \sum_{k=1}^m \langle T(X_\mathcal{H},X_i), Z_k\rangle Z_k f X_i f 
= \langle T(X_\mathcal{H},\nabla_\mathcal{H} f), \nabla_\mathcal{V} f\rangle.
\]
Collecting the above computations, we thus obtain
\[
-III=\langle \nabla_{\nabla_\mathcal{H} f} X_\mathcal{H}, \nabla_\mathcal{H} f\rangle+
\langle \nabla_{\nabla_\mathcal{H} f} X_\mathcal{V}, \nabla_\mathcal{V} f\rangle+ \langle T(X_\mathcal{H},\nabla_\mathcal{H} f), \nabla_\mathcal{V} f\rangle.
\]

It remains to compute $IV$. Similarly as above we have
\[
IV= \sum_{k=1}^m [X,Z_k]f Z_k f=
\sum_{k=1}^m (\nabla_X Z_k-\nabla_{Z_k}X-T(X,Z_k))f Z_k f.
\]
First observe that $\sum_{k=1}^m (\nabla_X Z_k)f Z_k f=0$ since $\nabla$ is metric and and the covariant derivative of vertical vector fields is vertical. Next due to the fact that $X_\mathcal{H}$ is basic then $\sum_{k=1}^m (\nabla_{Z_k}X)f Z_k f= \langle \nabla_{\nabla_\mathcal{V} f} X_\mathcal{V}, \nabla_\mathcal{V} f \rangle$.  Finally the torsion $T(X,Z_k)$ is zero. We conclude that 
\[
-\varepsilon IV= \varepsilon \langle \nabla_{\nabla_\mathcal{V} f} X_\mathcal{V}, \nabla_\mathcal{V} f \rangle.
\]
This completes the proof.
\end{proof}


Theorem \ref{CD} can be applied to any Carnot group of step 2.

A Carnot group of step two is a simply connected Lie group $\mathbb G$ whose Lie algebra $\mathfrak g$ can be written as $\mathfrak g=V_1\oplus V_2$, where $[V_1,V_1]=V_2$ and $[V_1,V_2]=\{0\}$. Denote by $e_1, \cdots, e_n$ an orthonormal basis of $V_1$ and by $\varepsilon_1, \cdots, \varepsilon_m$ an orthonormal basis of $V_2$. Let $X_1,\cdots,X_n$ and $Z_1, \cdots, Z_m$ be the corresponding left-invariant vector fields on $\mathbb G$. Then 
\[
X_i=\frac{\partial}{\partial x_i} -\frac12\sum_{k=1}^m\sum_{j=1}^n \gamma_{ij}^k x_j Z_k,
\]
where $\gamma_{ij}^k=\langle [e_i,e_j],\varepsilon_k\rangle$ are the group constants. We also have 
\[
[X_i,X_j]=\sum_{k=1}^m \gamma_{ij}^k  Z_k.
\]

Actually, it is known  (see \cite{BLU}) that such a Carnot group is isomorphic to 
 $\dR^N=\dR^n \times \dR^m$  equipped with the group law given by 
\[
(x,z) \cdot (x',z')=\left(x+x', z+z' + \frac{1}{2} \langle Bx,x' \rangle\right),
\]
where $x,x' \in \dR^d$, $z,z'\in \dR^m$ and 
\[
\langle Bx,x' \rangle = \left(\langle B^{(1)} x,x' \rangle, \cdots, \langle B^{(m)} x,x' \rangle \right) 
\]
 for  some linearly independent skew-symmetric $d\times d$ matrices $B^{(l)}$,  $1\leq l \leq m$. 
 In this case, we clearly have: 
 \[
 [X_i,Z_k]=0  \textrm{ and } [Z_j,Z_k]= 0.
 \]

The left invariant sub-Laplacian on $\mathbb G$ is $\Delta_\mathcal H=\sum_{i=1}^n X_i^2$. The dilations $\delta_t: \mathfrak g \to \mathfrak g$, $t\ge 0$, are defined  by scalar multiplication  $t^i$ on $V_i$. The dilations $\delta_t: \mathbb G\to\mathbb G$ are defined as $\delta_t (\exp Z)=\exp \delta_t(Z)$, for any $Z\in \mathfrak g$. The generator of the one-parameter  group $(\delta_{e^s})_{s\in \mathbb R}$ can be written as $D=\sum_{i=1}^n x_iX_i+2\sum_{k=1}^m z_k Z_k$. Similarly to \cite{LP}, we consider the operator $L=\Delta_{\mathcal H}-D$.

An interesting fact for  Carnot groups is that  the Bott connection is ``trivial'' in  the orthonormal basis $(X_1,X_n,Z_1,\dots Z_m)$; i.e. for all $1 \leq i,j \leq n, 1\leq k,l \leq m$
 \[
 \nabla _{X_i} X_j=\nabla_{X_i} Z_k =\nabla_{Z_k} X_i =\nabla_{Z_k} Z_l=0.
 \]

In the present setting we have for $f,g \in C^{\infty}(\mathbb G)$,
\[
\Gamma (f,f)=\sum_{i=1}^n (X_i f)^2, \quad \Gamma^Z(f,f)=\sum_{k=1}^m (Z_k f)^2.
\]
Denote the $\Gamma_2$ operators associated with the sub-Laplacian by $\Gamma_2^{\mathcal H}$ and $\Gamma_2^{Z,\mathcal H}$. It was proved in \cite[Proposition 2.21]{BG} that 
\[
\Gamma_2^{\mathcal H}(f,f)+\varepsilon \Gamma_2^{Z,\mathcal H}(f,f) \ge \frac1n(\Delta_\mathcal H f)^2-\frac{\kappa}{\varepsilon} \Ga(f,f)+\rho_2 \Ga^Z(f,f),
\]
where 
\[
\kappa=\sup_{\|x\|=1} \sum_{j=1}^n \sum_{k=1}^m \left(\sum_{i=1}^n \ga_{ij}^k x_i\right)^2, \quad 
\rho_2=\frac14\ \inf_{\|z\|=1} \sum_{i,j=1}^n  \left(\sum_{k=1}^m \ga_{ij}^k z_k\right)^2.
\]

We now turn to the curvature  dimension criterion  satisfied by $L=\Delta_{\mathcal H}-D$.

\begin{proposition}\label{CD-Carnot}
Let $f \in C^\infty(\mathbb G)$ and $\varepsilon >0$, then
\[
\Ga_2 (f,f) + \ve \Ga_2^Z (f,f) \geq \left(1-\frac{\kappa}{\ve}\right) \Ga(f,f) + 
\left(2\ve+ \rho_2\right)\Ga^Z(f,f).
\]
\end{proposition}
\begin{proof}
Note that $X_{\mathcal H}=\sum_{i=1}^n x_iX_i$ and $X_{\mathcal V}=2\sum_{k=1}^m z_k Z_k$. We compute
\[
\langle \nabla_{\nabla_\mathcal{V} f} X_\mathcal{V}, \nabla_\mathcal{V} f \rangle
= 2\left\langle \sum_{k,l=1}^m (Z_k f) \nabla_{Z_k} (z_l Z_l), \sum_{k=1}^m(Z_k f) Z_k \right\rangle =2  \Ga^Z(f,f).
\]

\[
\langle \nabla_{\nabla_\mathcal{H} f} X_\mathcal{H}, \nabla_\mathcal{H} f \rangle
= \left\langle \sum_{i,j=1}^n (X_i f) \nabla_{X_i} (x_j X_j), \sum_{i=1}^n(X_i f) X_i \right\rangle = \Ga(f,f).
\]
We now show that :  $  T(X_\mathcal{H} ,\nabla_\mathcal{H} f), \nabla_\mathcal{V} f \rangle + \langle \nabla_{\nabla_\mathcal{H} f} X_\mathcal{V}, \nabla_\mathcal{V} f \rangle=0$. Indeed, 
\begin{align*}
\langle \nabla_{\nabla_\mathcal{H} f} X_\mathcal{V}, \nabla_\mathcal{V} f \rangle
&= 2\left\langle \sum_{i=1}^n\sum_{l=1}^m (X_i f) \nabla_{X_i} (z_l Z_l), \sum_{k=1}^m(Z_k f) Z_k \right\rangle 
\\ &=
-\sum_{k=1}^m \sum_{i,j=1}^n \gamma_{ij}^k x_j(X_i f) (Z_k f)
\end{align*}

\begin{align*}
\langle T(X_\mathcal{H} ,\nabla_\mathcal{H} f), \nabla_\mathcal{V} f \rangle
&= -\left\langle \sum_{i,j=1}^n x_i (X_j f)[X_i,X_j], \sum_{k=1}^m(Z_k f) Z_k \right\rangle 
\\ &=
-\sum_{k=1}^m \sum_{i,j=1}^n \gamma_{ij}^k x_i(X_j f) (Z_k f)=
\sum_{k=1}^m \sum_{i,j=1}^n \gamma_{ij}^k x_j(X_i f) (Z_k f).
\end{align*}
Collecting the above computations, we conclude the proof from Theorem \ref{CD}.
\end{proof}

Proposition \ref{CD-Carnot} is obtained similarly to Theorem \ref{CD}. A more direct proof, as in the Heisenberg case, is also possible, since:
\[
[X_i,D] =X_i
\] and thus
\[ [L,D]=L .
\]
%
%
%

We note that if $X=\sum_{i=1}^n a_i X_i + \sum_{k=1}^m b_k Z_k$ for some smooth functions $a_i,b_k$ on such a Carnot group, then the horizontal part of $X$ is basic if and only if for all $1\leq i \leq n, 1\leq k\leq m$,
\[
Z_k (a_i)=0;
\] 
that is, the functions $a_i$ only depends on $x_1,\dots, x_n$. In this case, we have
\[
\langle \nabla_{\nabla_\mathcal{H} f} X_\mathcal{H}, \nabla_\mathcal{H} f \rangle
= \sum_{i,j=1}^n (X_i f) ({X_i} a_j) (X_j f),
\]

\[
\langle \nabla_{\nabla_\mathcal{V} f} X_\mathcal{V}, \nabla_\mathcal{V} f \rangle
= \sum_{k,l=1}^m (Z_k f)  (Z_k b_l)(Z_l f),
\]

\begin{align*}
\langle \nabla_{\nabla_\mathcal{H} f} X_\mathcal{V}, \nabla_\mathcal{V} f \rangle
&=  \sum_{i=1}^n\sum_{l=1}^m (X_i f) ({X_i}b_l)(Z_l f)
\end{align*}
 and 
\begin{align*}
\langle T(X_\mathcal{H} ,\nabla_\mathcal{H} f), \nabla_\mathcal{V} f \rangle
&= \left\langle \sum_{i,j=1}^n a_j(X_i f)[X_i,X_j], \sum_{l=1}^m(Z_l f) Z_l \right\rangle 
= \sum_{i=1}^n\sum_{l=1}^m   (X_i f)  \left( \sum_{j=1}^n a_j \gamma_{ij}^l \right) (Z_l f).
\end{align*}
\section{Convergence in $L^2$ and Poincar\'e inequalities} \label{sec:L2}

In this section we study the $L^2$ convergence to equilibrium of the semigroup $Q_t$ under the assumptions of subsection 2.1. Our basic assumption is the following:
\begin{equation}\label{e:d2-int}
D:=\int_\M \int_\M d^2(x,y) d\mu(x) d\mu(y) <+\infty,
\end{equation}
where $d$ is the subelliptic distance associated to $\Gamma$. We will prove that under this assumption, the semigroup $Q_t=e^{tL}$ converges in $L^2$ with an explicit exponential rate of convergence.

We also denote $d_\ve$ the distance associated to the ``carr\'e du champ'' operator   $\Gamma_\ve:= \Gamma +  \ve \Gamma^Z$.
Since $d_\ve\leq d$, condition \eqref{e:d2-int} is also satisfied for $d_\ve$:
\[
D_\ve:=\int_\M \int_\M d_\ve^2(x,y) d\mu(x) d\mu(y) <+\infty.
\]
We  recall the  generalized CD condition:
\[
\Gamma_2(f) +\ve \Gamma_2^Z (f) \ge \left( \rho_1-\frac{\kappa}{\ve}\right) \Gamma(f) + \left(\rho_2 +\rho_3 \ve \right) \Gamma^Z(f).
\]

\subsection{Convergence in $L^2$}

In the sequel, for $\ve > \frac{\rho_1}{\kappa}$, we denote
\[
\lambda_\ve = \min \left\{  \rho_1-\frac{\kappa}{\ve} , \frac{\rho_2}{\ve} +\rho_3   \right\} .
\]
Therefore,
\begin{equation} \label{e:Gamma2complet}
\Gamma_2(f) +\ve \Gamma_2^Z (f) \ge \lambda_\ve  \left(  \Gamma(f) +  \ve \Gamma^Z(f) \right) .
\end{equation}
The basic inequality is the following:

\begin{proposition}\label{prop:dec-gammaZ-exp}
Let $f\in C_0^{\infty}(\M)$ and $\ve>0$. Then for $t>0$ and $x\in \M$ one has
\[
\Gamma(Q_t f)(x) +  \ve \Gamma^Z(Q_t f)(x) \le e^{-2 \lambda_\ve t}  \left( Q_t\Gamma(f)(x) +  \ve Q_t\Gamma^Z(f)(x) \right).
\]
\end{proposition}


\begin{proof}
We follow the proof in \cite[Theorem 7.3]{B16ems} (see also \cite{W12}). In order to use the standard $\Gamma$-calculus, it suffices to prove that the heat semigroup $Q_t$ maps smooth functions of compact support into  smooth Lipschitz functions.  

 Fix $t>0$ and $n\ge 1$. Let $f\in C^{\infty}(\M)$ be compactly supported in the set $\{W\le n\}$. Consider 
\[
\Phi_n(s)=Q_s^n   \left(\Gamma(Q_{t-s}^n f) +  \ve \Gamma^Z(Q_{t-s}^n f)  \right).
\]
For the sake of convenience, we denote $g=Q_{t-s} f$ and $g_n=Q_{t-s}^n f$. Then
\[
\frac{\partial}{\partial s} \Phi_n(s)=Q_s^n \left(\left(L_n \Gamma(g_n)-2 \Gamma(g_n, L_n g_n)\right)+  \ve \left(L_n \Gamma^Z(g_n) -\Gamma^Z(g_n,L_ng_n) \right)\right).
\]
From the Cauchy-Schwarz inequality, 
\begin{align*}
\Gamma(g_n, L_n g_n)
&=
h_n^2L g_n\Gamma(g_n,  \log h_n)+h_n^2\Gamma(g_n,  L g_n)
\\ &\le 
\frac{1}{2} \left( \|L f\|_{L^{\infty}}^2 \Gamma(\log h_n)+\Gamma(g_n)\right)+h_n^2\Gamma(g_n,  L g_n).
\end{align*}
Since $h_n$ is supported in the set $\{W\le 2n\}$, then
\[
\Gamma(\log h_n)=\left(\frac{1}{nh_n}h'\left(\frac Wn\right)\right)^2 \Gamma(W) \le \frac{C}{h_n^2},
\]
and hence
\[
L_n \Gamma(g_n)-2 \Gamma(g_n, L_n g_n) \ge 2h_n^2 \Gamma_2(g_n)-h_n^2 \Gamma(g_n)-C,
\]
where the constant $C$ depends on $f$ and $t$, but does not depend on $n$. Similarly we obtain
\[
L_n \Gamma^Z(g_n)-2 \Gamma^Z(g_n, L_n g_n) \ge 2h_n^2 \Gamma_2^Z(g_n)-h_n^2 \Gamma^Z(g_n)-C.
\]
By a direct computation,
\[
L_n\left(\frac1{h_n^2} \right) \le \frac{C}{h_n^2},
\]
and as a consequence
\[
Q_s^n\left(\frac1{h_n^2} \right)\le \frac{e^{Cs}}{h_n^2}.
\]
Collecting the above estimates and applying the curvature dimension condition \eqref{e:Gamma2complet}, one has
\begin{align*}
\frac{\partial}{\partial s} \Phi_n(s)&\ge Q_s^n \left(2h_n^2\left(\Gamma_2(g_n)+\ve \Gamma_2^Z(g_n)\right)-h_n^2\left(\Gamma(g_n)+\ve \Gamma^Z(g_n)\right)-(1+\ve)C \right)
\\ &\ge
(2\lambda_{\ve}-1) \Phi_n(s)-(1+\ve)C.
\end{align*}
Integrating this inequality from 0 to $t$ yields
\[
\Gamma(Q_t^n f)+\ve\Gamma^Z(Q_t^n f)\le C, 
\]
where the constant $C$ depends on $f$ and $t$ is uniform on the set $\{W\le n\}$.

Now for any $x,y \in \M$ and $f\in C_0^{\infty}(\M)$, we pick $n$ big enough such that $x,y\in \{W\le n\}$ and $\mathrm{supp} (f)\subset \{W\le n\}$. It follows from the previous estimate that 
\[
|Q_t^n f(x)-Q_t^n f(y)|\le C d(x,y).
\]
Taking the limit of $n\to \infty$, one has
\[
|Q_t f(x)-Q_t f(y)|\le C d(x,y).
\]
We conclude that $Q_t$ transfers $C_0^{\infty}(\M)$ into a subset of the set of smooth Lipschitz functions. Therefore we can 
proceed by applying the Bakry-\'Emery machinery as in \cite{B16ems,W12} to justify the computations below

Let $f\in C_0^{\infty}(\M)$ and $t>0$, then
\begin{align*}
\frac{\partial}{\partial s} Q_s  \left( \Gamma(Q_{t-s} f) +  \ve \Gamma^Z(Q_{t-s} f)  \right)
   & = 2 Q_s  \left( \Gamma_2(Q_{t-s} f) +  \ve \Gamma_2^Z(Q_{t-s} f)  \right)\\
   &\geq 2 \lambda_\ve  Q_s  \left( \Gamma(Q_{t-s} f) +  \ve \Gamma^Z(Q_{t-s} f)  \right),
\end{align*}
and the result follows by Gronwall Lemma.  

\end{proof}

One then deduces the following:

\begin{corollary}\label{cor:CVL2}
For $f \in C_0^\infty (\M)$, one has $Q_t f \to \int_\M f d\mu$ in $L^2(\M,\mu)$ when $t \to +\infty$.
\end{corollary}

\begin{proof}
From Proposition \ref{prop:dec-gammaZ-exp},
\[
\Gamma(Q_t f) +  \ve \Gamma^Z(Q_t f) \le e^{-2 \lambda_\ve t}  \left( \| \Gamma(f) \|_\infty +  \ve \| \Gamma^Z(f) \|_\infty \right).
\]
Thus, integrating along some  geodesic from $x$ to $y$ for $d_\ve$ gives: 
\[
| Q_t f (x) -Q_t f(y)| \le e^{- \lambda_\ve t}  \left( \| \Gamma(f) \|_\infty +  \ve \| \Gamma^Z(f) \|_\infty \right)^{1/2} d_\ve (x,y).
\]
One deduces
\begin{align*}
\Var_\mu (Q_t f)  & :=\frac{1}{2}  \int_\M \int_\M ( Q_t f (x) -Q_t f(y))^2 d\mu(x) d\mu(y) \\
 & \le  \frac{1}{2} e^{-2 \lambda_\ve t}  \left( \| \Gamma(f) \|_\infty +  \ve \| \Gamma^Z(f) \|_\infty \right)  \int_\M \int_\M d_\ve (x,y)^2 d\mu(x) d\mu(y)\\
 &\le  \frac{D_\ve}{2} e^{-2 \lambda_\ve t}  \left( \| \Gamma(f) \|_\infty +  \ve \| \Gamma^Z(f) \|_\infty \right).
\end{align*}

Noting that  $\Var_\mu (Q_t f) = \Vert Q_t f -\int f d\mu \Vert_2^2$ since $\int Q_t f d\mu =\int f d\mu$, we conclude the proof by letting $t \to +\infty$.
\end{proof}

\subsection{Poincar\'e inequalities and quantitative estimates}

We have the following Poincar\'e inequality for the heat kernel measure:
\begin{proposition} \label{prop:Poinc-Z}
Let  $f \in C_0^\infty (\M)$, one has
\[
Q_t(f^2 )-(Q_t f)^2 \leq \frac{1-e^{-2\lambda_\ve t}} {\lambda_\ve} Q_t\left( \Gamma(f)+ \ve \Gamma^Z (f) \right) 
\]
and thus
\[
\int_\M f^2 d\mu-\left(\int_\M  f d\mu\right)^2 \leq \frac{1} {\lambda_\ve} \int_\M \left( \Gamma(f)+ \ve \Gamma^Z (f) \right) d\mu.
\]
\end{proposition}

\begin{proof}
The proof of the first inequality is given by the standard $\Gamma$-calculus. Indeed, for any $f\in C_0^\infty(\M)$ one has:
\begin{align*}
Q_t(f^2) - (Q_tf)^2 &= \int_0^t \frac{\partial}{\partial s} Q_s \left( (Q_{t-s} f)^2 \right) ds\\
                                     &= 2\int_0^t  Q_s \left( \Gamma (Q_{t-s} f) \right) ds\\
                                      &\leq 2 \int_0^t  Q_s \left( \Gamma (Q_{t-s} f) + \ve \Gamma^Z(Q_{t-s}f)  \right) ds\\
                                         &\leq 2  \int_0^t          e^{-2 \lambda_\ve (t-s)} ds \; Q_t  \left( \Gamma(f) +  \ve \Gamma^Z(f) \right)\\
                                         &= \frac{1-e^{-2\lambda_\ve t}} {\lambda_\ve} Q_t\left( \Gamma(f)+ \ve \Gamma^Z (f) \right) ,
\end{align*}
where we used Proposition \ref{prop:dec-gammaZ-exp} in the last inequality. The second inequality is obtained by letting $t\to +\infty$ and applying Corollary \ref{cor:CVL2}.
\end{proof}

\begin{corollary}\label{cor:dec-integree}
Let $f\in L^2(\M)$ be smooth such that $\Gamma(f), \Gamma^Z(f)\in L^1(\M)$. Then one has
\[
\int_\M (Q_tf)^2 d\mu-\left(\int_\M  f d\mu\right)^2 \leq \frac{e^{-2 \lambda_\ve t}} {\lambda_\ve} \int_\M \left( \Gamma(f)+ \ve \Gamma^Z (f) \right) d\mu.
\]
\end{corollary}
\begin{proof}
By an approximation argument, we may assume that $f\in  C_0^\infty (\M)$.
Applying  the second inequality of Proposition \ref{prop:Poinc-Z} to $Q_tf$ instead of $f$ gives:
\begin{align*}
\int_\M (Q_tf)^2 d\mu-\left(\int_\M  f d\mu\right)^2 
&\leq 
\frac{1} {\lambda_\ve} \int_\M \left( \Gamma(Q_t f)+ \ve \Gamma^Z (Q_t f) \right) d\mu
\\ &\leq 
\frac{e^{-2 \lambda_\ve t}} {\lambda_\ve} \int_\M \left( Q_t \Gamma(f)+ \ve Q_t \Gamma^Z (f) \right) d\mu
\\ &\leq 
\frac{e^{-2 \lambda_\ve t}}{\lambda_\ve} \int_\M \left(\Gamma(f)+ \ve \Gamma^Z (f) \right) d\mu.
\end{align*}
Here in the second inequality we used the exponential decay of Proposition \ref{prop:dec-gammaZ-exp}.
\end{proof}

We now turn to the exponential convergence in $L^2(\mu)$.
\begin{proposition}\label{prop-cv-exp-L2}
Let $f\in L^2(\mu)$.
Let $t\geq \frac {1}{2\lambda_\ve}$ and $C:= e\left( 1+ \frac{2\lambda_\ve\ve}{\rho_2}\right)\left(1+\frac{2\kappa}{\rho_2} \right)$.
Then
\[
 \left\Vert Q_t f -\int_\M f d\mu \right\Vert_{L^2(\mu)}^2 \leq  C  e^{- 2 \lambda_\ve t}   \left\Vert  f -\int_\M f d\mu \right\Vert_{L^2(\mu)}^2
\]
\end{proposition}


Before proceeding with the proof, we recall  the usual generalized curvature-dimension inequality $CD(\rho_1,\rho_2,\kappa,\infty)$ introduced in \cite{BG}:
\[
\Gamma_2(f) +\ve \Gamma_2^Z (f) \ge \left( \rho_1-\frac{\kappa}{\ve}\right) \Gamma(f) +\rho_2  \Gamma^Z(f).
\]
In our framework, $CD(\rho_1,\rho_2,\kappa,\infty)$  also holds since $\rho_3$ in assumption $(\bf{A3})$ is positive. Hence we can apply  the  reverse Poincar\'e and log-Sobolev inequalities obtained in \cite[Propositions 3.1 and 3.2]{BB}. We point out that even if the symmetry of the operator is assumed in  \cite{BB}, the symmetry is not used in the proofs of these propositions which are of a ``local nature''. The reverse Poincar\'e inequality states as follows.

\begin{proposition}[Reverse Poincar\'e inequality]\label{rev-poinc}
Let  $f \in C_0^\infty(\bM)$, then
for $x\in \M$,  $t>0$ one has
\[
 \Gamma(Q_t f)(x)  +\rho_2 t   \Gamma^Z( Q_t f)(x) \le \frac{1}{2t} \left(1+\frac{2\kappa}{\rho_2} \right) \big[Q_t ( f^2 )(x) -Q_t f(x)^2 \big].
\]

\begin{proof}[Proof of proposition \ref{prop-cv-exp-L2}]
By an approximation argument, we may assume that $f\in   C_0^\infty (\M)$.
Let $t\geq s>0$. Applying Corollary \ref{cor:dec-integree}  and the reverse Poincar\'e inequality  in Proposition \ref{rev-poinc} give
\begin{align*}
 \left\Vert Q_t f -\int_\M f d\mu \right\Vert_{L^2(\mu)}^2 &= 
 \int_\M (Q_tf)^2 d\mu-\left(\int_\M  f d\mu\right)^2 \\
&\leq 
\frac{e^{-2 \lambda_\ve (t-s)}} {\lambda_\ve} \int_\M \left( \Gamma(Q_sf)+ \ve \Gamma^Z (Q_sf) \right) d\mu\\
& \leq 
e^{-2 \lambda_\ve t} \frac{e^{2 \lambda_\ve s}} { 2 \lambda_\ve s } \left( 1+ \frac{\ve}{\rho_2s}\right)\left(1+\frac{2\kappa}{\rho_2} \right)  \int_\M \left(Q_s ( f^2 )(x) -Q_sf(x)^2\right) d\mu(x) \\
&  \leq  
e^{-2 \lambda_\ve t} \frac{e^{2 \lambda_\ve s}} { 2 \lambda_\ve s } \left( 1+ \frac{\ve}{\rho_2s}\right)\left(1+\frac{2\kappa}{\rho_2} \right)  \left[\int_\M   f^2 d\mu - \left(\int_\M f d\mu \right)^2 \right],
\end{align*}
where the last inequality follows from the Cauchy-Schwarz inequality.
The result then follows by letting $s=\frac {1}{2\lambda_\ve}$.
\end{proof}

\end{proposition}

\section{Entropic convergence, log-Sobolev inequalities and hypercontractivity} \label{sec:entropic}

The study of the entropic convergence to equilibrium is more difficult. In this section,  in order to prove the desired convergence,   we assume the following exponential integrability of the distance:
\[
E_{c_0,d}:=\int_\M \int_\M e^{c_0 d^2(x,y) }d\mu(x) d\mu(y) <+\infty,
\]
for some $c_0>0$.

\subsection{Hypercontractivity and Entropic convergence}

Denote $C_b^{\infty}(\M)=C^{\infty}(\M)\cap L^\infty(\M)$. For $\delta>0$ let $\mathcal A_\delta$ be the set of functions $f\in C_b^{\infty}(\M)$ such that $f=g+\delta$ for some $g \in C_b^{\infty}(\M) $, $g \ge 0$, such that $\Gamma(g), \Gamma^{Z}(g)\in L^1(\mu)$.

We recall first  the reverse log-Sobolev inequality in \cite[Proposition 3.1]{BB}.

\begin{proposition}[Reverse log-Sobolev inequality]\label{prop:reverse_logsob}
Let $\delta>0$ and   $f \in\mathcal A_\delta$. 
 For $x\in \M$,  $t>0$ one has
\[ Q_t f(x) \Gamma(\ln Q_t f)(x)  +\rho_2 t  Q_t f(x) \Gamma^Z(\ln Q_t f)(x) \le \frac{1}{t} \left(1+\frac{2\kappa}{\rho_2}\right) \big[Q_t ( f\ln f )(x) -Q_tf(x)\ln Q_t f(x)\big].
\]
\end{proposition}

As it is now well known, the reverse log-Sobolev inequality implies the Wang Harnack inequality.
This was first observed by F.Y. Wang \cite{W1} in a Riemannian framework and see Proposition 3.4 of \cite{BB} in our framework:

\begin{proposition}\label{prop:wang-harnack}[Wang Harnack  inequality]\\
Let $\alpha>1$. For $f \in L^\infty(\bM)$,   $f \ge 0$, $t>0$, $x,y \in \bM$,
$$
(Q_tf)^\alpha (x) \leq Q_t(f^\alpha) (y) \exp \left(  \frac{\alpha}{\alpha-1} \left(\frac{ 1+\frac{2\kappa}{\rho_2} }{4t}\right) d^2(x,y)\right).
$$
\end{proposition}

The following log-Harnack inequality follows easily from Wang Harnack inequality (see \cite[Proposition 3.5]{BB}).
\begin{proposition}
Let $f \in L^\infty(\bM)$ such that  $\inf f > 0$. Then for $t>0$, $x,y \in \bM$,
$$
Q_t(\ln f) (x) \leq \ln Q_t(f) (y) + \left(\frac{ 1+\frac{2\kappa}{\rho_2} }{4t}\right) d^2(x,y).
$$
\end{proposition}
%

Proposition \ref{prop:wang-harnack} also implies the superconctractivity (also called hyperboundedness).

\begin{proposition} \label{prop:hyperboundedness}[Hyperboundedness]\\
Let $\beta>\alpha >1$. Let  $C:= 1+\frac{2\kappa}{\rho_2}$ and  
\[N_t:= \int_\M \int_\M \exp\left( \frac{\beta}{\alpha-1} \frac{Cd^2(x,y)}{t}\right) d\mu(y)   d\mu (x). \]
Then for $f\in L^{\alpha}(\mu)$
\[
\Vert Q_t f \Vert_\beta \leq N_t^{1/\beta} \Vert  f \Vert_\alpha
\]
with $N_t <+\infty$ for $t>\frac{\beta C}{(\alpha-1) c_0}$ and $N_t\to 1$ when $t\to +\infty$. 
\end{proposition}

\begin{proof}
First observe that  $N_t\le E_{c_0,d} <+\infty$ for $t>\frac{\beta C}{(\alpha-1) c_0}$ and $N_t\to 1$ when $t\to +\infty$.

Now take $f\ge 0$ such that $\int f^\alpha d\mu=1$. From Proposition \ref{prop:wang-harnack}, dividing first and then taking integral with respect to $y$ yields,
\[
(Q_t f)^\alpha (x) \int_\M \exp\left( -\frac{\alpha}{\alpha-1} \frac{Cd^2(x,y)}{t} \right)d\mu (y)  \leq \int_\M Q_t( f^\alpha) (y) d\mu(y)=1
\]

Now for $\beta>\alpha>1$ and $t>\frac{\beta C}{(\alpha-1) c_0}$, integrating in $x$ gives
\begin{align*}
\int_\M (Q_t f)^\beta (x ) d\mu(x) & \leq  \int_\M \frac{1} { \left(\int_\M   \exp\left( - \frac{\alpha}{\alpha-1} \frac{C d^2(x,y)}{t}\right) d\mu(y)  \right)^{\beta/\alpha} } d\mu (x) \\
                          &  \leq  \int_\M  \left(\int_\M  \exp\left( \frac{\alpha}{\alpha-1} \frac{Cd^2(x,y)}{t}\right) d\mu(y) \right)^{\beta/\alpha} d\mu (x) \\
                          &  \leq  \int_\M \int_\M  \exp\left( \frac{\beta}{\alpha-1} \frac{Cd^2(x,y)}{t}\right) d\mu(y)   d\mu (x),
\end{align*}
where the second inequality follows from the fact that
\[
1 \leq \int_\M \frac{1}{g} d\mu  \int_\M g d\mu. 
\]
The above estimate writes
\[
\Vert Q_t f \Vert_\beta \leq N_t^{1/\beta} \Vert  f \Vert_\alpha
\]
and we conclude the proof. 
\end{proof}

As noticed by F.Y. Wang \cite{W4}, the above hyperboundedness property implies  the following entropic  convergence:
\begin{proposition} \label{prop:entropic-CV}
There exist $C,\theta >0$ such that for all $f\in L^1(\M)$, $f>0$, and all $t>0$
\begin{equation}\label{e:cv-entr}
\Ent_\mu( Q_t f) \leq C e^{-\theta t} \Ent_\mu (f).
\end{equation}
\end{proposition}

\begin{proof}
By Proposition \ref{prop:hyperboundedness},  for $T$ big enough one has
\[
\Vert Q_T f \Vert_{L^4(\mu)} \leq K\Vert  f \Vert_{L^2(\mu)}
\]
with \[
K< 2^{1/4}. \]

Denote $\mu(f)=\int_\M f d\mu$. According to F.Y.  Wang  in \cite[Proposition 2.2]{W4}, the above implies that for any $f\in L^2(\mu)$ with $\mu(f^2)\le1$,
\begin{equation}\label{e:cvL2}
\Vert Q_Tf-\mu(f) \Vert_{L^2(\mu)} \leq M<1
\end{equation}
and that  for $T_2$ big enough $Q_{T_2}$ is hypercontractive in the sense:
\begin{equation}\label{e:hyperc}
\Vert Q_{T_2} f \Vert_{L^4(\mu)} \leq  \Vert  f \Vert_{L^2(\mu)}.
\end{equation}

Now  the Riesz-Thorin interpolation theorem   implies the following entropy decay   (see e.g. Proposition 2.3 of Wang  \cite{W4}): for $f>0$   with $\mu(f)=1$,
\begin{equation}\label{eq:Ent decay}
\mu(Q_ {T_2} f \ln Q_{T_2} f) \leq \frac{2}{3} \mu(f \ln f).
\end{equation}
Iterating \eqref{eq:Ent decay} and applying Lemma \ref{lem:phi}
imply the following entropic convergence: for any $t>0$,
\[
\Ent_\mu( Q_t f) \leq  \frac{3}{2} \left(\frac{2}{3} \right)^\frac{t}{T_2} \Ent_\mu (f).
\]
Hence we conclude the proof.
\end{proof}

\begin{remark}
One  notes that inequality \eqref{e:cvL2}, together with  Lemma \ref{lem:phi}, implies the following $L^2$ convergence:
\begin{equation}\label{e:cvL2-bis}
\Vert Q_t f-\mu(f) \Vert_{L^2(\mu)} \leq C \exp(-\theta t) \Vert  f-\mu(f) \Vert_{L^2(\mu)}
\end{equation}
for some $C$ and $\theta>0$.
\end{remark}

\subsection{Log-Sobolev inequalities and quantitative estimates}
We can now provide some quantitative estimates.
Recall that the relation
\begin{equation}\label{e:GammaZ}
\Gamma (f, \Gamma^Z (f,f))= \Gamma^Z (f, \Gamma (f,f))
\end{equation}
is satisfied. The ``standard'' $\Gamma$-calculus applies.
In particular, the curvature criterion \eqref{e:Gamma2complet} implies:
\begin{proposition}\label{prop:dec-gammaZ-ln}
Let $\delta>0$ and $f\in A_{\delta}$.  Then for $t>0$,
\[
Q_t f \left(\Gamma( \ln Q_t f) +  \ve \Gamma^Z(\ln Q_t f)\right)) \le e^{-2 \lambda_\ve t} Q_t \left( f(\Gamma(\ln f) +  \ve \Gamma^Z( \ln f)) \right)
\]
\end{proposition}
\begin{proof}
Let $f\in \mathcal A_{\delta}$. Denote $g(x,s)=Q_{t-s}f(x)$. By the chain rule, one has $L \ln g=Lg/g-\Gamma(\ln g)$. Thus, proceeding as before, (rigorous justification is identical)
\begin{align*}
&\frac{\partial}{\partial s} Q_s (Q_{t-s} f\Gamma(\ln Q_{t-s} f))
\\&\qquad=Q_sL(g \Gamma(\ln g))-Q_s(Lg\Gamma(\ln g))-2Q_s(g\Gamma(\ln g, Lg/g))
\\ &\qquad=
Q_s(gL \Gamma(\ln g))+2Q_s(\Gamma(g,\Gamma(\ln g)))- 2Q_s(g\Gamma(\ln g, Lg+\Gamma(\ln g)))
\\ &\qquad=
2Q_s(g \Gamma_2(\ln g)).
\end{align*}

Similarly, since  the condition \eqref{e:GammaZ} holds one has
\begin{align*}
&\frac{\partial}{\partial s} Q_s (Q_{t-s} f\Gamma^Z(\ln Q_{t-s} f))
\\&\qquad=Q_sL(g \Gamma^Z(\ln g))-Q_s(Lg\Gamma^Z(\ln g))-2Q_s(g\Gamma^Z(\ln g, Lg/g))
\\ &\qquad=
Q_s(gL \Gamma^Z(\ln g))+2Q_s(\Gamma(g,\Gamma^Z(\ln g)))- 2Q_s(g\Gamma(\ln g, Lg+\Gamma^Z(\ln g)))
\\ &\qquad=
Q_s(gL \Gamma^Z(\ln g))+2Q_s(\Gamma^Z(g,\Gamma(\ln g)))- 2Q_s(g\Gamma(\ln g, Lg+\Gamma^Z(\ln g)))
\\ &\qquad=
2Q_s(g \Gamma_2^Z(\ln g)).
\end{align*}

Combining the above two equalities yields
\begin{align*}
&\frac{\partial}{\partial s} Q_s  \left( Q_{t-s} f (\Gamma(\ln Q_{t-s} f) +  \ve \Gamma^Z(\ln Q_{t-s} f) ) \right)
\\&\qquad \quad=2 Q_s  \left( Q_{t-s} f (\Gamma_2(\ln Q_{t-s} f) +  \ve \Gamma_2^Z(\ln Q_{t-s} f) ) \right)
\\&\qquad \quad \geq 2 \lambda_\ve  Q_s  \left( Q_{t-s} f (\Gamma(\ln Q_{t-s} f) +  \ve \Gamma^Z(\ln Q_{t-s} f) ) \right).
\end{align*}
The  claim then follows from the Gronwall Lemma.
\end{proof}

\begin{proposition}\label{prop:LogS-Z}
Let $f\in C_{0}^{\infty}(\M)$, $f>0$ and $t>0$. Then
\begin{equation}\label{eq:QLogS-Z}
Q_t(f \ln f )-Q_t f \ln (Q_t f) \leq \frac{1-e^{-2\lambda_\ve t}} {2 \lambda_\ve} Q_t\left( \frac{\Gamma(f)}{f} + \ve \frac{\Gamma^Z(f)}{f} \right) 
\end{equation}
and thus
\begin{equation}\label{eq:LogS-Z}
\int_\M f \ln f d\mu-\int_\M fd\mu   \ln \left(\int_\M f d\mu \right) \leq \frac{1} {2 \lambda_\ve} \int_\M \left( \frac{\Gamma(f)}{f} + \ve \frac{\Gamma^Z(f)}{f} \right) d\mu.
\end{equation}
\end{proposition}

\begin{proof}
The proof of the first inequality is given by the standard $\Gamma$-calculus and is similar to the proof of Proposition \ref{prop:Poinc-Z}. Indeed, for any $\delta>0$ let $g\in \mathcal A_{\delta}$. Then from the chain rule,
\begin{align*}
\frac{\partial}{\partial s} Q_s (Q_{t-s}g\ln Q_{t-s} g)
&=LQ_s(Q_{t-s} g \ln (Q_{t-s} g))-Q_s(LQ_{t-s} g\ln (Q_{t-s} g))-Q_s(LQ_{t-s} g)
\\ &=
Q_sL(Q_{t-s} g \ln (Q_{t-s} g))-Q_s(LQ_{t-s} g\ln (Q_{t-s} g))-Q_s(LQ_{t-s} g)
\\ &=
Q_s \left(Q_{t-s}g \Gamma (\ln Q_{t-s} g) \right).
\end{align*}
Therefore one has:
\begin{align*}
Q_t(g \ln g ) - (Q_tg) \ln (Q_t g) &= \int_0^t \frac{\partial}{\partial s} Q_s \left( Q_{t-s} g \ln (Q_{t-s} g) \right) ds\\
                                     &= \int_0^t  Q_s \left(Q_{t-s}g \Gamma (\ln Q_{t-s} g) \right) ds\\
                                      &\leq  \int_0^t  Q_s \left( Q_{t-s}g\left( \Gamma (\ln Q_{t-s} g) + \ve \Gamma^Z(\ln Q_{t-s}g)  \right) \right) ds\\
                                         &\leq   \int_0^t          e^{-2 \lambda_\ve (t-s)} ds \; Q_t  \left(g (\Gamma(\ln g) +  \ve \Gamma^Z(\ln g)) \right)\\
                                         &= \frac{1-e^{-2\lambda_\ve t}} {2\lambda_\ve}  Q_t  \left(g (\Gamma(\ln g) +  \ve \Gamma^Z(\ln g)) \right),
\end{align*}
where we used Proposition \ref{prop:dec-gammaZ-ln} in the last inequality. Note that, as in Lemma \ref{lem:phi},  the computation of the above derivative only requires the diffusion property of the generator and not its  symmetry. 

Now for any $f\in C_{0}^{\infty}(\M)$ such that $f>0$, consider $g=f+\delta\in \mathcal A_{\delta}$. Letting $\delta\to 0$, the previous estimate then yields \eqref{eq:QLogS-Z}.
Letting $t\to +\infty$ in \eqref{eq:QLogS-Z} and applying Corollary \ref{prop:entropic-CV}, we conclude the second inequality \eqref{eq:LogS-Z}.
\end{proof}

\begin{corollary}\label{cor:dec-ent-integree}
Let $f\in C_{0}^{\infty}(\M)$, $f>0$ and $t>0$. Then
\[
\Ent_\mu (Q_tf) \leq \frac{e^{-2 \lambda_\ve t}} {2\lambda_\ve} \int_\M \left( \frac{\Gamma(f)}{f} + \ve \frac{\Gamma^Z(f)}{f} \right) d\mu.
\]
\end{corollary}
\begin{proof}
Applying  \eqref{eq:LogS-Z} in Proposition \ref{prop:LogS-Z} to $Q_tf$ instead of $f$ gives:
\[
\Ent_\mu (Q_tf) \leq \frac{1} {2\lambda_\ve} \int_\M \left( Q_t f \Gamma(\ln Q_t f)+ \ve Q_t f\Gamma^Z (\ln Q_t f) \right) d\mu.
\]
Then Proposition \ref{prop:dec-gammaZ-ln} concludes the proof.
\end{proof}

By arguing as in Proposition \ref{prop-cv-exp-L2}, we obtain the following  entropic convergence.
\begin{proposition}\label{prop-cv-entropic}
Let $f\in C_0^{\infty}(\M)$ such that $f>0$.
Let $t\geq \frac {1}{2\lambda_\ve}$ and $C:= e\left( 1+ \frac{2\lambda_\ve\ve}{\rho_2}\right)\left(1+\frac{2\kappa}{\rho_2} \right)$.
Then
\[
 \Ent_\mu (Q_t f) \leq  C  e^{-2 \lambda_\ve t} \Ent_\mu ( f).
\]
\end{proposition}

\begin{proof}
Let $t\geq s>0$.  Applying first Corollary \ref{cor:dec-ent-integree}  and then Proposition \ref{prop:reverse_logsob}, one has for $f>0$ 
\begin{align*}
\Ent_\mu (Q_t f) 
&\leq \frac{e^{-2 \lambda_\ve (t-s)}} {2\lambda_\ve}\int_\M \left( Q_s f \Gamma(\ln Q_s f)+ \ve Q_s f\Gamma^Z (\ln Q_s f) \right) d\mu
\\  & \leq 
e^{-2 \lambda_\ve t} \frac{e^{2 \lambda_\ve s}} { 2 \lambda_\ve s } \left( 1+ \frac{\ve}{\rho_2s}\right)\left(1+\frac{2\kappa}{\rho_2} \right)  \int_\M \left(Q_s( f\ln f ) -Q_sf\ln Q_sf \right)d\mu 
\\& \leq 
e^{-2 \lambda_\ve t} \frac{e^{2 \lambda_\ve s}} { 2 \lambda_\ve s } \left( 1+ \frac{\ve}{\rho_2 s}\right)\left(1+\frac{2\kappa}{\rho_2} \right) \left( \int_\M  f\ln f  -   \int_\M f d\mu \ln\left(\int_\M f d\mu\right) \right),
\end{align*}
where the last inequality follows from Jensen's inequality.
We conclude the proof by taking $s=\frac {1}{2\lambda_\ve}$.
\end{proof}
\bibliographystyle{plain}
 \bibliography{OU_Refs}

\end{document}